\documentclass[12pt]{amsart}

\usepackage{amsmath,amssymb,amsfonts,amsthm}
\usepackage{enumitem}

\usepackage{subcaption}
\usepackage{graphicx}

\makeatletter
\@namedef{subjclassname@2020}{%
	\textup{2020} Mathematics Subject Classification}
\makeatother

\newtheorem{theorem}{Theorem}[section]

\newtheorem{lemma}[theorem]{Lemma}

\theoremstyle{definition}
\newtheorem{definition}[theorem]{Definition}

\newtheorem{example}[theorem]{Example}

\AtBeginDocument{{\noindent\small
This is a preprint of a paper whose final and definite form is with 
\emph{Appl. Math. (Warsaw)}, ISSN: 1233-7234. 
Article accepted for publication 11-June-2024.} 
\vspace{9mm}}

\begin{document}
	
\baselineskip=17pt	

\title[Uniform Stability of Dynamic SICA HIV Transmission Models]{%
Uniform Stability of Dynamic SICA HIV Transmission Models on Time Scales}


\author[Z. Belarbi]{Zahra Belarbi}

\address{Zahra Belarbi\newline
\indent Department of Mathematics,\newline
\indent University Mustapha Stambouli of Mascara,
B.P. 305 Mascara, Algeria;\newline
\indent Ecole Normale Sup\'{e}rieure, Mostaganem, Algeria}
\email{zahra.belarbi@univ-mascara.dz}


\author[B. Bayour]{Benaoumeur Bayour}
\address{Benaoumeur Bayour\newline
\indent Department of Mathematics,\newline
\indent University Mustapha Stambouli of Mascara,
B.P. 305 Mascara, Algeria}
\email{b.bayour@univ-mascara.dz}


\author[D. F. M. Torres]{Delfim F. M. Torres}
\address{Delfim F. M. Torres\newline
\indent Center for Research and Development in Mathematics\newline
\indent and Applications (CIDMA), Department of Mathematics,\newline 
\indent University of Aveiro, 3810-193 Aveiro, Portugal;\newline
\indent Research Center in Exact Sciences (CICE),\newline
\indent Faculty of Sciences and Technology,\newline
\indent University of Cape Verde (Uni-CV),
7943-010 Praia, Cabo Verde}
\email{delfim@ua.pt; delfim@unicv.cv\newline 
\indent ORCID: http://orcid.org/0000-0001-8641-2505}


\date{}


\begin{abstract}
We consider a SICA model for HIV transmission on time scales.
We prove permanence of solutions and we derive sufficient
conditions for the existence and uniform asymptotic stability
of a unique positive almost periodic solution of the system
in terms of a Lyapunov function.
\end{abstract}



\maketitle


\section{Introduction}

In 2015, the deterministic SICA model was first presented
as a sub-system of a TB-HIV/AIDS co-infection model \cite{[Silva]}.
One of the primary objectives of SICA models is to demonstrate
how some of the fundamental relationships between epidemiological
variables and the general pattern of the AIDS epidemic can be clarified
using a straightforward mathematical model \cite{May}. The celebrated
SICA mathematical model \cite{MyID:471,[Torres2],[Lotfi],[Torres],MyID:481}
divides the total human population into four compartments, namely
\begin{itemize}
\item $S(t)$: susceptible individuals at time $t$;
\item $I(t)$: HIV-infected individuals with no clinical symptoms
of AIDS but able to transmit HIV to other individuals at time $t$;
\item $C(t)$: HIV-infected individuals under antiretroviral therapy (ART),
the so called chronic stage with a viral load remaining low at time $t$;
\item $A(t)$: HIV-infected individuals with AIDS clinical symptoms at time $t$.
\end{itemize}
Under some realistic assumptions, the dynamics of the disease proliferation
in a community is then translated into a mathematical model given by the following system
of four ordinary differential equations \cite{[Silva],[Silva2],[Silva3]}:
\begin{eqnarray}
\label{1}
\left\{
\begin{array}{l l}
\dot{S}(t)=\Lambda-\beta \lambda(t)S(t)-\nu S(t),\\
\dot{I}(t)=\beta \lambda(t) S(t)-(\rho+\phi+\nu)I(t)+\gamma A(t)+\omega C(t),\\
\dot{C}(t)=\phi I(t)-(\omega+\nu)C(t),\\
\dot{A}(t)=\rho I(t)-(\gamma+\nu+d)A(t),
\end{array}
\right.
\end{eqnarray}
where $\Lambda$, $\beta$, $\nu$, $\rho$, $\phi$, $\gamma$, $\omega$
and $d$ are real positive rates,
\begin{itemize}
\item $\Lambda$ is the rate of new susceptible;
\item $\beta$ is the HIV transmission rate;
\item $\nu$ is the natural death rate;
\item $\rho$ is the default treatment rate for $I$ individuals;
\item $\phi$ is the HIV treatment rate for $I$ individuals;
\item $\gamma$ is the AIDS treatment rate;
\item $\omega$ is the default treatment rate for $C$ individuals;
\item $d$ is the AIDS induced death rate;
\end{itemize}
and where the effective contact rate with
people infected with HIV is given by
$$
\lambda(t)=\frac{\beta}{N(t)}(I(t)+\eta_C C(t)+\eta_A A(t))
$$
with
\begin{itemize}
\item $N(t)$ the total population at time $t$, that is,
$$
N(t) = S(t) + I(t) + C(t) + A(t);
$$
\item $0\leq\eta_C \leq 1$ the modification parameter;
\item $\eta_A \geq 1$ the partial restoration parameter
of immune function of individuals with HIV infection
that use correctly the ART treatment.
\end{itemize}

The study of dynamical systems on time scales is now a very
active area of research. The books of Bohner and Peterson
\cite{[b1],[b0]} offer a good introduction with applications
to the time scale calculus along with some advanced topics.
Applications of the time scale calculus can be found in many areas,
including economics \cite{[Martin2],MR3313737,[Tisdell]},
ecology \cite{[khudd],[khudd2],MR3124383} and epidemics
\cite{[Martin],MR4185230,MR4378676}.
Here we generalize the SICA model \eqref{1} by considering
dynamic equations on time scales and study it using
the time-scale theory. By doing it, we unify the continuous
and discrete-time models \cite{MyID:481}, generalizing it also
to other contexts like the quantum \cite{MR1865777,MR3184533}
or mixed/hybrid settings \cite{MR3525237,MR4252546}.

Recently, Prasad and Khuddush proved the existence and uniform
asymptotic stability of positive and almost periodic solutions
for a 3-species Lotka--Volterra competitive system on time scales \cite{[khudd2]}.
Moreover, they also studied the permanence and positive almost periodic solutions
of a $n$-species Lotka--Volterra system on time scales \cite{[khudd]}.
Motivated by these works, here we investigate the permanence and uniform asymptotic
stability of the unique positive almost periodic solution
of the following SICA model on time scales:
\begin{eqnarray}
\label{1.1}
\left\{ \begin{array}{l l}
x_{1}^{\Delta}(t)=\Lambda-\beta \lambda(t)x_{1}^{\sigma}(t)-\nu x_{1}^{\sigma}(t),\\
x_{2}^{\Delta}(t)=\beta \lambda(t) x_{1}(t)-(\rho+\phi+\nu)x_{2}^{\sigma}(t)
+\gamma x_{4}(t)+\omega x_{3}(t),\\
x_{3}^{\Delta}(t)=\phi x_{2}(t)-(\omega+\nu)x_{3}^{\sigma}(t),\\
x_{4}^{\Delta}(t)=\rho x_{2}(t)-(\gamma+\nu+d)x_{4}^{\sigma}(t),
\end{array}
\right.
\end{eqnarray}
where $t\in\mathbb{T}^{+}$, with $\mathbb{T}^{+}$ a nonempty closed subset of
$\mathbb{R}^{+}=]0,+\infty[$.

Note that, in our notation,
$(x_{1}(t),x_{2}(t),x_{3}(t),x_{4}(t))$ is interpreted as $(S(t),I(t),C(t),A(t))$.

The paper is organized as follows.
Section~\ref{sec2} gives a brief introduction to the time-scale
theory and we also recall there the definition of an almost periodic function
with respect to time, uniformly for the state variables, and
the lemma of Lyapunov that gives a sufficient condition for the
existence of a unique almost periodic solution that is uniformly asymptotically stable.
In Section~\ref{sec3}, we derive sufficient conditions for system (\ref{1.1}) to be permanent.
Finally, in Section~\ref{sec4}, we establish sufficient conditions for the existence
and uniform asymptotic stability of the unique positive almost periodic solution
to the SICA system (\ref{1.1}). Our results seem to be new even for standard
time scales. We end with an example (Example~\ref{ex01})
and a brief conclusion (Section~\ref{sec5}).


\section{Some preliminaries}
\label{sec2}

For more background on the theory of time scales
than the one given here we recommend the books
\cite{[b1],[b0]}. A time scale $\mathbb{T}$
is a nonempty closed subset of the real numbers
$\mathbb{R}$. The time scale $\mathbb{T} $ has the topology that it inherits
from the real numbers with the standard topology. Let $f$ be a function
defined on $ \mathbb{T}^{+}$. We set
$$
f^{L} = \inf \left\lbrace f(t): t \in \mathbb{T}^{+}\right\rbrace
\quad \text{and} \quad
f^{U} = \sup \left\lbrace f(t): \, t \in \mathbb{T}^{+}\right\rbrace.
$$

\begin{definition}[See \cite{[b1]}]
The jump operators $\sigma, \rho: \, \mathbb{T} \longrightarrow \mathbb{T} $
and the graininess function $ \mu: \mathbb{T}\longrightarrow \mathbb{R}^{+} $
are defined, respectively, by
$$
\sigma(t)=\inf\left\lbrace \tau \in \mathbb{T}: \tau >t\right\rbrace , \;
\rho(t)=\sup \left\lbrace \tau \in \mathbb{T}: \tau < t\right\rbrace, \;
\text{ and } \; \mu(t)= \sigma(t)-t
$$
supplemented by  $\inf\emptyset := \sup \mathbb{T}$
and  $\sup \emptyset := \inf \mathbb{T}$.
A point $ t \in \mathbb{T} $ is left-dense, left-scattered,
right-dense, or right-scattered if $ \rho(t)=t$, $\rho(t)<t$,
$\sigma(t)=t$, or $\sigma(t)>t$, respectively.
If $\mathbb{T} $ has a left-scattered maximum $M$, then
we define $\mathbb{T}^{\kappa}=\mathbb{T} \setminus \left\lbrace M \right\rbrace$;
otherwise, $ \mathbb{T}^{\kappa}=\mathbb{T} $. If $\mathbb{T} $
has a right-scattered minimum $ m $, then we define
$ \mathbb{T}_{\kappa} =\mathbb{T} \setminus \left\lbrace m \right\rbrace$;
otherwise, $ \mathbb{T}_{\kappa}=\mathbb{T}$. If $f:\mathbb{T}\rightarrow \mathbb{R}$,
then $f^{\sigma}:\mathbb{T}\rightarrow \mathbb{R}$ is given by
$f^{\sigma}(t)=f(\sigma(t))$ for all $t\in\mathbb{T}$.
\end{definition}

\begin{definition}
Let $f:\mathbb{T}\rightarrow \mathbb{R}$ and
$t\in\mathbb{T}^{\kappa}$. We  define $f^{\Delta}(t)$ to be the number,
provided it exists, with the property that given any
$\varepsilon>0$ there is a neighborhood $U$ of $t$
(i.e., $U=(t-\delta,t+\delta)\cap\mathbb{T}$ for some $\delta>0$) such that
$$
\left|[f^\sigma(t)-f(s)]-f^{\Delta}(t)[\sigma(t)-s]\right|
\leq\varepsilon|\sigma(t)- s| \mbox{ for all  } \,s\in U.
$$
We call $f^{\Delta}(t)$ the delta derivative
of $f$ at $t$. Moreover, we say that $f$ is delta (or Hilger) 
differentiable on $\mathbb{T}^{\kappa}$ provided $f^{\Delta}(t)$ 
exists for all $t\in\mathbb{T}^{\kappa}$.
\end{definition}

\begin{definition}[See \cite{[b1]}]
A function $f: \mathbb{T}\longrightarrow \mathbb{R}$ 
is said to be regressive provided
$1+ \mu(t)f(t)\neq 0 $ for all $ t \in \mathbb{T}^{\kappa}$.
The set of all regressive and rd-continuous functions
(a function $ g: \mathbb{T}\longrightarrow \mathbb{R} $ is called rd-continuous
provided it is continuous at right-dense points in $ \mathbb{T}$
and its left-sided limits exist (finite) at left-dense points in
$ \mathbb{T}$) $ f: \mathbb{T}\longrightarrow \mathbb{R} $
will be denoted by $ \mathcal{R}=\mathcal{R}(\mathbb{T},\mathbb{R})$.
We define the set 
$$ 
\mathcal{R}^{+}= \mathcal{R}^{+}(\mathbb{T},\mathbb{R})
=\left\lbrace f \in \mathcal{R}: 1+\mu(t)f(t)>0
\text{ for all } t \in \mathbb{T}\right\rbrace.
$$
\end{definition}
Let now $p$ and $q\in\mathcal{R}$.
We define the circle plus addition $\oplus$ on $\mathcal{R}$ by
$$
(p\oplus q)(t):=p(t)+q(t)+\mu(t)p(t)q(t)
\quad \text{for all}\quad t\in\mathbb{T}
$$
and the circle minus subtraction $\ominus$ on $\mathcal{R}$ by
$$
(p\ominus q)(t):= \frac{p(t)-q(t)}{1+\mu(t)q(t)}
\quad \text{for all}\quad t\in\mathbb{T}.
$$
The time scales exponential function $e_{p}(\cdot,t_{0})$
is defined for $p\in\mathcal{R}$ and $t_{0}\in\mathbb{T}$
as the unique solution of the initial value problem
$$
y^{\Delta}=p(t)y \,\,\text{and} \,\,
y(t_{0})=1 \,\,\text{on}\,\,\mathbb{T}.
$$
One has
$$
e_{p}(\cdot,t_{0})e_{q}(\cdot,t_{0})
=e_{p\oplus q}(\cdot,t_{0}),
\ \frac{e_{p}(\cdot,t_{0})}{e_{q}(\cdot,t_{0})}
=e_{p\ominus q}(\cdot,t_{0})
\ \text{and} \ 
e_{\ominus q}(\cdot,t_{0})=\frac{1}{e_{q}(\cdot,t_{0})}.
$$

\begin{lemma}[See \cite{[b4]}]
\label{l2.3}
Assume that $a >0$, $b>0$, and that $-a \in \mathcal{R}^{+}$.
Then,
$$
y^{\Delta}(t)\geq (\leq) \ b-a y^{\sigma}(t), \; y(t)>0,
\; \; t \in [t_{0},\infty)_{\mathbb{T}}
$$
implies that
$$
y(t) \geq (\leq) \ \dfrac{b}{a} \left[ 1
+\left( \dfrac{a y(t_{0})}{b}-1\right)e_{(\ominus a)}(t,t_{0})\right], \;
t \in [t_{0}, \infty)_{\mathbb{T}}.
$$
\end{lemma}

\begin{definition}[See \cite{[Li]}]
A time scale $ \mathbb{T} $ is called an almost periodic time scale if
$$
\Pi= \left\lbrace \tau \in \mathbb{R}: \,
t+\tau \in \mathbb{T} \text{ for all } \,
t \in \mathbb{T}\right\rbrace  \neq \left\lbrace 0 \right\rbrace.
$$
\end{definition}

\begin{definition}[See \cite{[Li]}]
Let $ \mathbb{T} $ be an almost periodic time scale.
A function $ x \in C( \mathbb{T},\mathbb{R}^{n} )  $
is called an almost periodic function if the
$\varepsilon$-translation set of $ x $,
$$
E\left\lbrace \varepsilon, x\right\rbrace
=\left\lbrace \tau \in \Pi :
\mid x(t+\tau)-x(t)\mid
< \varepsilon  \text{ for all }
t \in \mathbb{T} \right\rbrace,
$$
is a relatively dense set in $ \mathbb{T} $, that is,
for all $ \varepsilon >0 $ there exists a
constant $ l(\varepsilon)>0 $ such that each interval of length
$ l(\varepsilon)$ contains a
$\eta(\varepsilon) \in E \left\lbrace \varepsilon,x\right\rbrace$
such that $ \lvert x(t+\tau)-x(t)\lvert < \varepsilon$
for all $ t \in \mathbb{T}$.
Moreover, $ \tau $ is called the $\varepsilon$-translation
number of $x(t)$ and $l(\varepsilon) $ is called the inclusion
length of $ E\left\lbrace \varepsilon,x\right\rbrace$.
\end{definition}

\begin{definition}[See \cite{[Li]}]
Let $ \mathbb{D} $ be an open set in $ \mathbb{R}^{n} $
and let $ \mathbb{T} $ be a positive almost periodic time scale.
A function $ f \in C(\mathbb{T}\times \mathbb{D},\mathbb{R}^{n}) $
is called an almost periodic function in $ t \in \mathbb{T} $,
uniformly for $ x \in \mathbb{D} $, if the $ \varepsilon $-translation
set of $f$,
$$
E \left\lbrace \varepsilon, f, \mathbb{S}\right\rbrace
= \left\lbrace \tau \in \Pi: \mid f(t+\tau)-f(t)\mid <\varepsilon
\, \text{ for all } \, (t,x)\in \mathbb{T}\times \mathbb{S} \right\rbrace,
$$
is a relatively dense set in $ \mathbb{T} $ for all  $ \varepsilon>0 $
and for each compact subset $ \mathbb{S} $ of $ \mathbb{D}$
there exists a constant $ l(\varepsilon,\mathbb{S})>0 $
such that each interval of length $ l(\varepsilon,\mathbb{S}) $
contains $ \tau(\varepsilon,\mathbb{S}) \in E\left\lbrace \varepsilon,
f,\mathbb{S}\right\rbrace  $ such that
$$
\mid f(t+\tau,x)-f(t,x)\mid < \varepsilon \, \,
\text{ for all } \, \,
(t,x)\in \mathbb{T} \times \mathbb{S}.
$$
\end{definition}

Consider the system
\begin{equation}
\label{eqly}
x^{\Delta}(t)=h(t,x)
\end{equation}
and its associate product system
$$
x^{\Delta}(t)=h(t,x), \, \, z^{\Delta}(t)=h(t,z),
$$
where
$h: \mathbb{T}^{+} \times \mathbb{S}_{B}
\longrightarrow \mathbb{R}^{n}, \; \mathbb{S}_{B}
= \left\lbrace   x \in \mathbb{R}^{n} : \, \| x \| < B \right\rbrace $
and $ h(t,x) $ is almost periodic in $ t $ uniformly
for $ x \in \mathbb{S}_{B} $ and continuous in $ x $.

\begin{lemma}[See \cite{[b18]}]
\label{L2.2}
Suppose that there exists a Lyapunov function $V(t,x,z) $ defined on
$\mathbb{T}^{+} \times \mathbb{S}_{B}\times \mathbb{S}_{B} $, that is,
there exists a function $V(t,x,z) $ satisfying the following conditions:
\begin{enumerate}
\item $a(\| x-z\|) \leq V(t,x,z)\leq b(\| x-z \|),
\; \text{ where } \; a,b \in \mathbb{K}$ with
$$
\mathbb{K}=\left\lbrace  \alpha \in C(\mathbb{R}^{+},
\mathbb{R}^{+}): \alpha(0)=0,  \text{ and } \; \alpha \;
\text{ increasing } \right\rbrace;
$$

\item $| V(t,x,z)-V(t,x_{1},z_{1}) |
\leq L (\| x-x_{1}\| +\| z-z_{1}\|),$
where  $L >0$  is  a  constant;

\item $D^{+}V^{\Delta}(t,x,z)\leq -c V(t,x,z)$,
where $c>0$ and $-c \in \mathcal{R^{+}}$.
\end{enumerate}
Furthermore, if there exists a solution $ x(t) \in \mathbb{S} $
of system (\ref{eqly}) for $ t \in \mathbb{T^{+}} $,
where $ \mathbb{S}\cup \mathbb{S}_{B} $  is a compact set,
then there  exist a unique almost periodic solution
$ f(t) \in \mathbb{S}$ of system (\ref{eqly}),
which is uniformly asymptotically stable.
\end{lemma}

\begin{definition}[See \cite{[khudd]}]
System (\ref{1.1}) is said to be permanent
if there exist positive constants $m$ and $M$
such that 
$$ 
m\leq \liminf_{t \longrightarrow \infty} x_{i}(t)
\leq \limsup_{t \longrightarrow \infty}x_{i}(t)\leq M,
\quad i=1,2,3,4,
$$
for any solution $(x_{1}(t),x_{2}(t),x_{3}(t),x_{4}(t))$ of (\ref{1.1}).
\end{definition}


\section{Permanence of positives solutions}
\label{sec3}

The principal objective of this section is to establish
sufficient conditions for system (\ref{1.1}) to be permanent.
Let $ t_{0} \in \mathbb{T} $ be a fixed positive initial time.
We introduce the following assumption for (\ref{1.1}):

\begin{itemize}
\item[$(H_{1})$] $\lambda(t)$
is a bounded and almost periodic function
satisfying
$$
0< \lambda^{L}\leq \lambda(t) \leq \lambda^{U}.
$$
\end{itemize}

\begin{lemma}
\label{l3.1}	
Suppose hypothesis $(H_{1})$ holds. Then, for any positive solution
$(x_{1}(t),x_{2}(t),x_{3}(t),x_{4}(t))$ of system (\ref{1.1}),
there exists positive constants $ M$ and $T$ such that
$ x_{i}(t)\leq M$, $i=1,2,3,4$,  for $ t\geq T$.
\end{lemma}

\begin{proof}
Let $Z(t)=\left(x_{1}(t),x_{2}(t),x_{3}(t),x_{4}(t)\right)$
be any positive solution of system (\ref{1.1}). From the $i$th
equation of system (\ref{1.1}), we have
\begin{eqnarray}
\left
\{\begin{array}{l l}
x_{1}^{\Delta}(t)\leq \Lambda -(\beta \lambda^{L}+\nu)x_{1}^{\sigma}(t),\\
x_{2}^{\Delta}(t) \leq \beta \lambda^{U} M_{1}+(\gamma+\omega)
\dfrac{\Lambda}{\nu}-(\rho+\phi+\nu)x_{2}^{\sigma}(t),\\
x_{3}^{\Delta}(t)\leq \phi M_{2}-(\omega+\nu)x_{3}^{\sigma}(t),\\
x_{4}^{\Delta}(t)\leq \rho M_{2}-(\gamma+\nu+d)x_{4}^{\sigma}(t).
\end{array}
\right.
\end{eqnarray}
Hence, by Lemma~\ref{l2.3}, there exist positive constants $M_{i} $
and $ T_{i}$ such that for any positive solution
$(x_{1}(t),x_{2}(t),x_{3}(t),x_{4}(t))$  of system (\ref{1.1}),
we have
$$
x_{1}(t)\leq \dfrac{\Lambda}{\beta \lambda^{L}+\nu}
\left[ 1+\left( \dfrac{(\beta \lambda^{L}+\nu)
x_1(t_{0})}{\Lambda}-1\right)
e_{\ominus(\beta \lambda^{L}+\nu)}(t,t_{0}) \right].
$$
If $-(\beta \lambda^{L} + \nu)<0$, then
$e_{\ominus(\beta \lambda^{L}+\nu)}(t,t_{0})
\longrightarrow 0$ as $t \longrightarrow \infty$ and
\begin{eqnarray}
\begin{cases}
x_{1}(t) \leq M_1:= \Lambda/(\beta \lambda^{L}+\nu),
& \text{ for } \, t\geq T_{1},\\
x_{2}(t) \leq M_2:=\beta (\lambda^{U}M_{1}+(\gamma+\omega)
\dfrac{\Lambda}{\nu})/(\rho+\phi+\nu),
& \text{ for } \, t \geq T_{2},\\
x_{3}(t)\leq M_3:=\phi M_{2}/(\omega+\nu),
&  \text{ for } \, t\geq T_{3}, \\
x_{4}(t)\leq  M_4:=\rho M_{2}/(\gamma+\nu+d),
& \text{ for } \, t\geq T_{4}.
\end{cases}
\end{eqnarray}
Let $ M=\displaystyle{\max_{1\leq i \leq 4}}\left\lbrace M_{i}\right\rbrace$
and $ T=\displaystyle{\max_{1\leq i \leq 4}}\left\lbrace T_{i}\right\rbrace$.
Then, $ x_{i}(t)\leq M$, $i=1,2,3,4$, for all $ t \geq T$.
\end{proof}

\begin{lemma}
\label{lem1}
Suppose that $ (H_{1})$ holds. Then, system (\ref{1.1}) is permanent.
\end{lemma}

\begin{proof}
Let  $ Z(t)=(x_{1}(t),x_{2}(t),x_{3}(t),x_{4}(t))$
be any positive solution of system (\ref{1.1}).
From the $i$th equation of system (\ref{1.1}),
we have
\begin{eqnarray}
\left\{ \begin{array}{l l}
x_{1}^{\Delta}(t)\geq \Lambda -(\beta \lambda^{U}+\nu)x_{1}^{\sigma}(t),\\
x_{2}^{\Delta}(t) \geq \beta \lambda^{L} m_{1}-(\rho+\phi+\nu)x_{2}^{\sigma}(t),\\
x_{3}^{\Delta}(t)\geq \phi m_{2}-(\omega+\nu)x_{3}^{\sigma}(t),\\
x_{4}^{\Delta}(t)\geq  \rho m_{2}-(\gamma+\nu+d)x_{4}^{\sigma}(t).
\end{array}
\right.
\end{eqnarray}
From hypothesis $(H_{1})$ and  Lemma~\ref{l2.3}, there exists
positive constants $m_{i}>0 $  such that for any positive solution
$(x_1(t),x_2(t),x_3(t),x_4(t))$  of system (\ref{1.1})
there exists $ \hat{T_{i}} $ such that
\begin{eqnarray}
\begin{cases}
x_{1}(t)\geq m_1:=  \Lambda/(\beta \lambda^{U}+\nu),
& \text{ for } \, t\geq\hat{T}_1,\\
x_{2}(t) \geq m_2:= (\beta \lambda^{L}m_{1})/(\rho+\phi+\nu),
& \text{ for } \, t\geq\hat{T}_2,\\
x_{3}(t)\geq  m_3:=\phi m_{2}/(\omega+\nu),
& \text{ for } \, t\geq\hat{T}_3,\\
x_{4}(t)\geq  m_4:=\rho m_{2}/(\gamma+\nu+d),
& \text{ for } \, t\geq\hat{T}_4.
\end{cases}
\end{eqnarray}
Let $ m= \min_{1\leq i\leq 4}\left\lbrace m_{i}\right\rbrace  $
and $ \hat{T}=\max_{1 \leq i \leq 4}\left\lbrace \hat{T_{i}}\right\rbrace$.
We conclude that $ x_{i}(t) \geq m$, $i=1,2,3,4$, for all $ t \geq \hat{T}$.
\end{proof}


\section{Uniform asymptotic stability}
\label{sec4}

In this section, we prove sufficient conditions for the
existence and uniform asymptotic stability of the unique positive
almost periodic solution to system (\ref{1.1}). Let us define
\begin{multline*}
\Omega := \Biggl\{ Z(t)=(x_1(t),x_2(t),x_3(t),x_4(t))
\in(\mathbb{R}^{+})^{4}: (x_1(t),x_2(t),x_3(t),x_4(t))\\
\text{ is a solution of } (\ref{1.1}) \text{ with } 0<m\leq x_i\leq M,
i=1,\ldots4, \text{ and } N(t)\leq\frac{\Lambda}{\nu} \Biggr\}.
\end{multline*}
It is clear that $\Omega$ is an invariant set of system (\ref{1.1})
and, by Lemma~\ref{lem1}, we have $\Omega\neq\emptyset$.
We introduce some more notation. Let
\begin{equation*}
\begin{split}
a_1 &:=\beta\lambda^{L}+\nu,\\
a_2 &:=\rho+\phi+\nu,\\
a_3 &:=\omega+\nu,\\
a_4 &:=\gamma+\nu+d,\\
b_1 &:=\left(\beta\lambda^{U}+\frac{\beta^{2}M}{2m}\right),\\
b_2 &:=\left(\rho+\phi\right),\\
b_3 &:=\left(\omega+\frac{\beta^{2}M}{2m}\eta_C\right),\\
b_4 &:=\left(\gamma+\frac{\beta^{2}M}{2m}\eta_A\right).
\end{split}
\end{equation*}
Moreover, let $ \Gamma_{1}:= \displaystyle{\min_{1\leq i\leq 4}} a_{i}$
and $\Gamma_{2}:= \displaystyle{\max_{1\leq i\leq4}}b_{i}$.
In our next result (Theorem~\ref{last:thm}) we assume the following additional hypothesis:
\begin{itemize}
\item[$(H_{2})$] $ \Gamma_{2} < \Gamma_{1} $ with $\Gamma_1,\Gamma_2 \in\mathcal{R}^{+}$.
\end{itemize}

\begin{theorem}
\label{last:thm}
Suppose that $(H_1)$ and $(H_2)$ hold. Then the dynamic system (\ref{1.1})
has a unique almost periodic solution
$Z(t)=(x_{1}(t),x_{2}(t),x_{3}(t),x_{4}(t)) \in \Omega$
that is uniformly asymptotically stable.
\end{theorem}

\begin{proof}
According to Lemma~\ref{l2.3}, every solution 
$$ 
Z(t)=(x_{1}(t),x_{2}(t),x_{3}(t),x_{4}(t)) 
$$
of system (\ref{1.1}) satisfies $ x_{i}^{L} \leq x_{i}(t) \leq x_{i}^{U} , i=1,\ldots,4$,
and $|x_{i}| \leq K_{i}, i=1,\ldots,4 $. Denote
\begin{equation*}
\begin{split}
\parallel Z \parallel
&= \parallel (x_{1}(t),x_{2}(t),x_{3}(t),x_{4}(t))\parallel \\
&= \displaystyle{ \sup_{t \in \mathbb{T^{+}}}} \left( \mid x_{1}(t)\mid
+ \mid x_{2}(t)\mid + \mid x_{3}(t)\mid + \mid x_{4}(t)\mid\right) .
\end{split}
\end{equation*}
Let $ Z(t)=(x_{1}(t),x_{2}(t),x_{3}(t),x_{4}(t))$ and
$\hat{Z}(t) =(\hat{x}_{1}(t),\hat{x}_{2}(t),\hat{x}_{3}(t) ,\hat{x}_{4}(t))$
be two positive solutions of (\ref{1.1}). Then $ \parallel Z \parallel \leq K$
and $ \parallel \hat{Z}\parallel\leq K $, where
$$
K= \displaystyle{\sum_{i=1}^{i=4}}K_{i}.
$$
In view of (\ref{1.1}), we have
\begin{eqnarray*}
\left
\{\begin{array}{l l}
x_{1}^{\Delta}(t)=\Lambda-\beta \lambda(t)x_{1}^{\sigma}(t)-\nu x_{1}^{\sigma}(t),\\
x_{2}^{\Delta}(t)=\beta \lambda(t) x_{1}(t)-(\rho+\phi+\nu)x_{2}^{\sigma}(t)
+\gamma x_{4}(t)+\omega x_{3}(t),\\
x_{3}^{\Delta}(t)=\phi x_{2}(t)-(\omega+\nu)x_{3}^{\sigma}(t),\\
x_{4}^{\Delta}(t)=\rho x_{2}(t)-(\gamma+\nu+d)x_{4}^{\sigma}(t),
\end{array}
\right.
\end{eqnarray*}
and
\begin{eqnarray*}
\left
\{\begin{array}{l l}
\hat{x}_{1}^{\Delta}(t)
=\Lambda-\beta \hat{\lambda}(t)\hat{x}_{1}^{\sigma}(t)-\nu\hat{x}_{1}^{\sigma}(t),\\
\hat{x}_{2}^{\Delta}(t)=\beta\hat{\lambda}(t) \hat{x}_{1}(t)
-(\rho+\phi+\nu)\hat{x}_{2}^{\sigma}(t)+\gamma\hat{x}_{4}(t)+\omega\hat{x}_{3}(t),\\
\hat{x}_{3}^{\Delta}(t)=\phi\hat{x}_{2}(t)-(\omega+\nu)\hat{x}_{3}^{\sigma}(t),\\
\hat{x}_{4}^{\Delta}(t)=\rho\hat{x}_{2}(t)-(\gamma+\nu+d)\hat{x}_{4}^{\sigma}(t).
\end{array}
\right.
\end{eqnarray*}
Define the Lyapunov function $ V(t, Z, \hat{Z})$
on $\mathbb{T^{+}}\times \Omega \times \Omega$ as
$$
V(t,Z,\hat{Z})= \displaystyle{\sum_{i=1}^{i=4}}\mid x_{i}(t)
-\hat{x_{i}}(t)\mid= \displaystyle{\sum_{i=1}^{i=4}}V_i(t),
$$
where $V_i(t)=\mid x_{i}(t)-\hat{x_{i}}(t)\mid $.
The two norms
$$
\parallel Z(t)-\hat{Z} (t)\parallel
= \displaystyle{\sup_{t \in \mathbb{T^{+}}}}
\displaystyle{\sum_{i=1}^{i=4}}
\mid x_{i}(t)-\hat{x_{i}}(t)\mid
$$
and
$$
\parallel Z(t)-\hat{Z} (t)\parallel_{\ast}
=\displaystyle{\sup_{t \in \mathbb{T^{+}}}}
\left( \displaystyle{\sum_{i=1}^{i=4}}\left( x_{i}(t)
-\hat{x_{i}}(t)\right) ^{2}\right) ^{\frac{1}{2}}
$$
are equivalent, that is, there exist two constants
$\eta_{1}$ and $\eta_{2}>0$ such that
$$
\eta_{1}\parallel Z(t)-\hat{Z}(t)\parallel_{\ast}\leq\parallel Z(t)
-\hat{Z}(t)\parallel\leq\eta_{2}\parallel Z(t)-\hat{Z}(t)\parallel_{\ast}.
$$
Hence,
$$
\eta_{1}\parallel Z(t)-\hat{Z}(t)\parallel_{\ast}
\leq V(t,Z,\hat{Z})\leq\eta_{2}\parallel Z(t)
-\hat{Z}(t)\parallel_{\ast}.
$$
Let $a, b \in C(\mathbb{R^{+}},\mathbb{R^{+}})$,
$a(x)=\eta_{1}x$ and $b(x)=\eta_{2}x$.
Then the assumption (i) of Lemma~\ref{L2.2} is satisfied.
Moreover,
\begin{equation*}
\begin{split}
\mid V(t,Z,\hat{Z})- V(t,Z^{\ast},\hat{Z^{\ast}})\mid
&= \left\vert\displaystyle{\sum_{i=1}^{i=4}}
\mid x_{i}(t)-\hat{x_{i}}(t) \mid
- \displaystyle{\sum_{i=1}^{i=4}}
\mid x^{\ast}_{i}(t)-\hat{x^{\ast}_{i}}(t) \mid \right\vert \\
&\leq \displaystyle{\sum_{i=1}^{i=4}}
\mid\left(  x_{i}(t)-\hat{x_{i}}(t) \right)
-\left(  x^{\ast}_{i}(t)-\hat{x^{\ast}_{i}}(t)\right)\mid\\
&\leq \displaystyle{\sum_{i=1}^{i=4}}\mid\left(x_{i}(t)
-x^{\ast}_{i}(t)\right)+\left(  \hat{x^{\ast}}_{i}(t)
-\hat{x_{i}}(t)\right)\mid \\
&\leq \displaystyle{\sum_{i=1}^{i=4}}
\mid \left(  x_{i}(t)-x^{\ast}_{i}(t) \right) \mid
+  \displaystyle{\sum_{i=1}^{i=4}}
\mid  \left(  \hat{x^{\ast}}_{i}(t)-\hat{x_{i}}(t)\right) \mid \\
&\leq L \left( \parallel Z(t)-Z^{\ast}(t)\parallel
+ \parallel \hat{Z}(t)-\hat{Z^{\ast}}(t)\parallel\right),
\end{split}
\end{equation*}
where $ L=1 $, so that condition (ii) of Lemma \ref{L2.2} is also satisfied.
Now, let $ v_{i}(t)=x_{i}(t)-\hat{x_{i}}(t)$, $i=1, \ldots, 4$. 
We compute and estimate the Dini derivative $D^+V^\Delta$ of $V$ 
along the associated product system (\ref{1.1}).   
Using Lemma~4.1 of \cite{wong}, it follows that
$D^{+}V_1^{\Delta}(t) \leq sign(v_1^{\sigma}(t))(v_1(t))^{\Delta}$. 
For more details on  $D^+V_1^{\Delta}(t)$ see \cite{Hong,461}. 
Now, let us begin computing
\begin{eqnarray*}
D^{+}V_1^{\Delta}(t) &\leq& sign(v_1^{\sigma}(t))(v_1(t))^{\Delta} \\
&=& sign(v_1^{\sigma}(t))[-\beta \lambda(t)x_1^{\sigma}(t)
-\nu x_1^{\sigma}(t)+\beta\hat{\lambda}(t)
\hat{x}_1^{\sigma}(t)+\nu \hat{x}_1^{\sigma}(t)]\\
&=&sign(v_1^{\sigma}(t))[-\beta\lambda(t)(x_1^{\sigma}(t)
- \hat{x}_1^{\sigma}(t))-\nu(x_1^{\sigma}(t)- \hat{x}_1^{\sigma}(t))]\\
&&-\beta \hat{x}_1^{\sigma}(t)(\lambda(t)-\hat{\lambda}(t))]\\
&=&sign(v_1^{\sigma}(t))[-(\beta\lambda(t)+\nu)
(\hat{x}_1^{\sigma}(t)- \hat{x}_1^{\sigma}(t))\\
&&-\beta \hat{x}_1^{\sigma}(t)(\lambda(t)-\hat{\lambda}(t))]\\
&\leq& -(\beta\lambda(t)+\nu)\mid\hat{x}_1^{\sigma}(t)
- \hat{x}_1^{\sigma}(t))\mid+\beta \hat{x}_1^{\sigma}(t)
\mid\lambda(t)-\hat{\lambda}(t)\mid\\
&\leq& -(\beta\lambda^{L}(t)+\nu)\mid\hat{x}_1^{\sigma}(t)
- \hat{x}_1^{\sigma}(t))\mid
+\beta \hat{x}_1^{\sigma}(t)\mid\lambda(t)-\hat{\lambda}(t)\mid\\
&\leq& -(\beta\lambda^{L}(t)+\nu)\mid\hat{x}_1^{\sigma}(t)
- \hat{x}_1^{\sigma}(t))\mid+\beta M\mid\lambda(t)-\hat{\lambda}(t)\mid,
\end{eqnarray*}
that is,
\begin{eqnarray*}
D^{+}V_{1}^{\Delta}(t)
&\leq&
-(\beta\lambda^{L}(t)+\nu)\mid\hat{x}_1^{\sigma}(t)
- \hat{x}_1^{\sigma}(t)\mid+(\beta M)\mid\lambda(t)-\hat{\lambda}(t)\mid.
\end{eqnarray*}
Let $ v_{2}(t)= x_{2}(t)-\hat{x_{2}}(t)$. Similarly, we have
\begin{equation*}
\begin{split}
D^{+}&V_{2}^{\Delta}(t)
\leq sign(v_2^{\sigma}(t))(v_2(t))^{\Delta} \\
&= sign(v_{2}^{\sigma}(t))[\beta\lambda(t)x_1(t)
-(\rho+\phi+\nu)x_2^{\sigma}(t)+\gamma x_4(t)\\
&+\omega x_3(t)-\beta\hat{\lambda}(t)\hat{x}_1(t)
+(\rho+\phi+\nu)\hat{x}_2^{\sigma}(t)-\gamma \hat{x}_4(t)-\omega \hat{x}_3(t)]\\
&= sign(v_{2}^{\sigma}(t))[(\beta\lambda(t)x_1(t)
-\beta\hat{\lambda}(t)\hat{x}_1(t))-(\rho+\phi+\nu)(x_2^{\sigma}(t)
-\hat{x}_2^{\sigma}(t))\\
&+\gamma(x_4(t)-\hat{x}_4(t))+\omega(x_3(t)-\hat{x}_3(t))]\\
&= sign(v_{2}^{\sigma}(t))[\beta(\lambda(t)(x_1(t)
-\hat{x}_1(t))+\hat{x}_1(t)(\lambda(t)-\hat{\lambda}(t)))\\
&-(\rho+\phi+\nu)(x_2^{\sigma}(t)-\hat{x}_2^{\sigma}(t))
+\gamma(x_4(t)-\hat{x}_4(t))+\omega(x_3(t)-\hat{x}_3(t))]\\
&= sign(v_{2}^{\sigma}(t))[\beta\lambda(t)(x_1(t)-\hat{x}_1(t))
+\beta \hat{x}_1(t)(\lambda(t)-\hat{\lambda}(t))\\
&-(\rho+\phi+\nu)(x_2^{\sigma}(t)-\hat{x}_2^{\sigma}(t))
+\gamma(x_4(t)-\hat{x}_4(t))+\omega(x_3(t)-\hat{x}_3(t))]\\
&\leq \beta\lambda^{U}\mid x_1(t)-\hat{x}_1(t) \mid
+ \beta M \mid \lambda(t)-\hat{\lambda}(t)\mid\\
&- (\rho+\phi+\nu)\mid x_2^{\sigma}(t)-\hat{x}_2^{\sigma}\mid
+\gamma\mid x_4(t)-\hat{x}_4(t)\mid\\
&+\omega\mid x_3(t)-\hat{x}_3(t)\mid;
\end{split}
\end{equation*}
for $v_{3}(t)= x_{3}(t)-\hat{x_{3}}(t)$ one has
\begin{eqnarray*}
D^{+}V_{3}^{\Delta}(t) &\leq& sign(v_3^{\sigma}(t))(v_3(t))^{\Delta}\\
&=& sign(v_{3}^{\sigma}(t))[\phi x_2(t)-(\omega+\nu)x_3^{\sigma}(t)
-\phi \hat{x}_2(t)+(\omega+\nu)\hat{x}_3^{\sigma}(t)]\\
&\leq& \phi \mid x_2(t)-\hat{x}_2(t) \mid
- (\omega+\nu)\mid x_3^{\sigma}(t)-\hat{x}_3^{\sigma}(t) \mid;
\end{eqnarray*}
and for $v_{4}(t)= x_{4}(t)-\hat{x_{4}}(t)$ we have
\begin{eqnarray*}
D^{+}V_{4}^{\Delta}(t) & \leq & \rho\mid x_2(t)-\hat{x}_2(t)
\mid -(\gamma+\nu+d)\mid x_4^{\sigma}(t)-\hat{x}_4^{\sigma}(t) \mid.
\end{eqnarray*}
Since
$$
\lambda(t)=\frac{\beta}{N(t)}(x_1(t)+\eta_C x_3(t)+\eta_A)x_4(t),
$$
it follows that
\begin{equation}
\label{inéq1}
\begin{split}
D^{+}V^{\Delta}(t)
&\leq
-(\beta\lambda^{L}(t)+\nu)\mid\hat{x}_1^{\sigma}(t)
- \hat{x}_1^{\sigma}(t))\mid+(\beta M)\mid\lambda(t)-\hat{\lambda}(t)\mid\\
&\quad +\beta\lambda^{U}\mid x_1(t)-\hat{x}_1(t) \mid
+ \beta M \mid \lambda(t)-\hat{\lambda}(t)\mid\\
&\quad - (\rho+\phi+\nu)\mid x_2^{\sigma}(t)-\hat{x}_2^{\sigma}\mid
+\gamma\mid x_4(t)-\hat{x}_4(t)\mid\\
&\quad +\omega\mid x_3(t)-\hat{x}_3(t)\mid
+\phi \mid x_2(t)-\hat{x}_2(t) \mid\\
&\quad - (\omega+\nu)\mid x_3^{\sigma}(t)-\hat{x}_3^{\sigma}(t) \mid\\
&\quad +\rho\mid x_2(t)-\hat{x}_2(t) \mid -(\gamma+\nu+d)\mid x_4^{\sigma}(t)
-\hat{x}_4^{\sigma}(t) \mid.
\end{split}
\end{equation}
Therefore,
\begin{equation*}
\begin{split}
|\lambda(t)-\hat{\lambda}(t)|
&\leq\frac{\beta}{4m}\left[\mid x_1(t)-\hat{x}_1(t) \mid
+\eta_C\mid x_3(t)-\hat{x}_3(t) \mid\right.\\
&\left.\quad +\eta_A\mid x_4(t)-\hat{x}_4(t) \mid\right],
\end{split}
\end{equation*}
where $0\leq\eta_C\leq1$, $\eta_A\geq1$, and $\beta>0$.
The inequality (\ref{inéq1}) becomes
\begin{equation}
\label{ineq2}
\begin{split}
D^{+}V^{\Delta}(t)
&\leq-(\beta\lambda^{L}+\nu)\mid x_1^{\sigma}(t)
- \hat{x}_1^{\sigma}(t)\mid-(\rho+\phi+\nu)\mid x_2^{\sigma}(t)
- \hat{x}_2^{\sigma}(t)\mid\nonumber\\
&-(\omega+\nu)\mid x_3^{\sigma}(t)- \hat{x}_3^{\sigma}(t)\mid
-(\gamma+\nu+d)\mid x_4^{\sigma}(t)- \hat{x}_4^{\sigma}(t)\mid\nonumber\\
&+2\beta M\mid \lambda(t)-\hat{\lambda}(t)\mid+(\beta \lambda^{U})\mid x_1(t)
-\hat{x}_1(t) \mid+\gamma\mid x_4(t)-\hat{x}_4(t) \mid\nonumber\\
&+\omega\mid x_3(t)-\hat{x}_3(t) \mid
+(\rho+\phi)\mid x_2(t)-\hat{x}_2(t) \mid\nonumber \\
&=-\left(\beta\lambda^{L}+\nu\right)\mid x_1^{\sigma}(t)
-\hat{x}_1^{\sigma}(t) \mid-(\rho+\phi+\nu)\mid
x_2^{\sigma}(t)-\hat{x}_2^{\sigma}(t) \mid\nonumber\\
&-(\omega+\nu)\mid x_3^{\sigma}(t)
-\hat{x}_3^{\sigma}(t) \mid-(\gamma+\nu+d)\mid
x_4^{\sigma}(t)-\hat{x}_4^{\sigma}(t) \mid\nonumber\\
&+\beta\left(\lambda^{U}
+\frac{\beta M}{2m}\right)\mid x_1(t)-\hat{x}_1(t) \mid
+(\rho+\phi)\mid x_2(t)-\hat{x}_2(t) \mid\nonumber\\
&+\left(\omega+\frac{\beta^{2}M}{2m}\eta_C\right)\mid x_3(t)
-\hat{x}_3(t) \mid\nonumber\\
&+\left(\gamma+\frac{\beta^{2}M}{2m}\eta_A\right)\mid x_4(t)
-\hat{x}_4(t) \mid.
\end{split}
\end{equation}
From (\ref{ineq2}) we can write that
\begin{eqnarray*}
D^{+}V^{\Delta}(t)
&\leq&-(\beta\lambda^{L}+\nu)\mid x^{\sigma}_1(t)
-\hat{x}^{\sigma}_1(t)\mid-(\rho+\phi+\nu)\mid x_2^{\sigma}(t)
-\hat{x}_2^{\sigma}(t)\mid\\
&&-(\omega+\nu)\mid x^{\sigma}_3(t)-\hat{x}^{\sigma}_3(t)\mid-(\gamma+\nu+d)\mid x^{\sigma}_4(t)
-\hat{x}^{\sigma}_4(t)\mid\\
&&+\left(\beta\lambda^{U}+\frac{\beta^{2}M}{2m}\right)\mid\hat{x}_1(t)-\hat{x}_1(t)\mid\\
&&+(\phi+\rho)\mid x_2(t)-\hat{x}_2(t)\mid\\
&&+\left(\omega+\frac{\beta^{2}M}{2m}\eta_C\right)\mid x_3(t)-\hat{x}_3(t)\mid\\
&&+\left(\gamma+\frac{\beta^{2}M}{2m}\eta_A\right)\mid x_4(t)-\hat{x}_4(t)\mid\\
&=&-a_1\mid x^{\sigma}_1(t)-\hat{x}^{\sigma}_1(t)\mid-a_2\mid x_2^{\sigma}(t)
-\hat{x}_2^{\sigma}(t)\mid\\
&&-a_3\mid x^{\sigma}_3(t)-\hat{x}^{\sigma}_3(t)\mid-a_4\mid x^{\sigma}_4(t)
-\hat{x}^{\sigma}_4(t)\mid\\
&&+b_1\mid x_1(t)-\hat{x}_1(t)\mid+b_2\mid x_2(t)-\hat{x}_2(t)\mid\\
&&+b_3\mid x_3(t)-\hat{x}_3(t)|
+b_4\mid x_4(t)-\hat{x}_4(t)\mid\\
&=& -\Gamma_{1}V(\sigma(t))+\Gamma_{2}V(t)\\
&=& \left(\Gamma_{2}-\Gamma_{1}\right)V(t)-\Gamma_{1}\mu(t)D^{+}V^{\Delta}(t)
\end{eqnarray*}
and it follows that
$D^{+}V^{\Delta}(t)\leq \dfrac{\Gamma_{2}
-\Gamma_{1}}{ (1+\Gamma_{1}\mu)}V(t)=-\Psi V(t)$.
By hypothesis $(H_{2})$, one has
$-\Psi \in \mathcal{R^{+}}$ and
$\Psi=\dfrac{\Gamma_{1}-\Gamma_{2}}{ (1+\Gamma_{1}\mu)} >0$.
Thus, the assumption (iii) of Lemma~\ref{L2.2} is satisfied.
Furthermore, the conditions (i) and (ii) of Lemma~\ref{L2.2}
also hold. For condition (i) we consider two functions
$a,b\in C(\mathbb{R}^{+})$ with $a(x)= 0.5 x$ and $b(x)=2x$.
For condition (ii) we put $L=1$. So there exists a unique uniformly
asymptotically stable almost periodic solution
$Z(t)=(x_{1}(t),x_{2}(t),x_{3}(t),x_{4}(t))$
of the dynamic system (\ref{1.1}) with $ Z(t) \in \Omega$.
\end{proof}

\begin{example}
\label{ex01}
Motivated by \cite{[Lotfi]} and the case of Morocco,
we consider the following system for $\mathbb{T}=\mathbb{Z}^{+}$:
\begin{eqnarray}
\label{example}
\left\{ \begin{array}{l l}
x_{1}^{\Delta}(t)
=\Lambda-\beta \lambda(t)x_{1}^{\sigma}(t)-\nu x_{1}^{\sigma}(t),\\
x_{2}^{\Delta}(t)
=\beta \lambda(t) x_{1}(t)-(\rho+\phi+\nu)x_{2}^{\sigma}(t)
+\gamma x_{4}(t)+\omega x_{3}(t),\\
x_{3}^{\Delta}(t)
=\phi x_{2}(t)-(\omega+\nu)x_{3}^{\sigma}(t),\\
x_{4}^{\Delta}(t)
=\rho x_{2}(t)-(\gamma+\nu+d)x_{4}^{\sigma}(t),
\end{array}
\right.
\end{eqnarray}
where
$$
x_1(0) = 1-\frac{11}{N_0},
\quad x_2(0) = \frac{2}{N_0},
\quad x_3(0) = 0,
\quad x_4(0) = \frac{9}{N_0},
$$
with $N_0=325235$, $\Lambda=2190$, $\beta=2.7 10^{-7}$, $\nu=0.39$, $\rho=0.2$,
$\phi=0.1$, $\gamma=0.33$, $\omega=0.09$, and $d=1$.
The system \eqref{example} is permanent. Furthermore,
$m_1=5615.381462$, $m_2=5.478704120 10^{-9}$,
$m_3=1.141396692 10^{-9}$, $m_4=6.370586186 10^{-10}$,
$M_1=5615.384615$, $M_2=  0.002773104793$, $M_3=0.0005777301652$,
and $M_4= 0.0003224540457$. The conditions of Theorem~\ref{last:thm}
are verified with $0.37=\Gamma_2< \Gamma_1=0.39$ and
$\Psi=0.01391941151$. Therefore, system \eqref{example} has a unique
positive almost periodic solution, which is uniformly asymptotic stable.
In Figure~\ref{fig} we plot the solution for the first 7 years
with $\eta_C = 0.5$ and $\eta_A = 1.5$. Computationally, it is difficult
to plot the solution for longer periods of time.
\begin{figure}
\centering
\begin{subfigure}{0.44\textwidth}
\includegraphics[width=\textwidth]{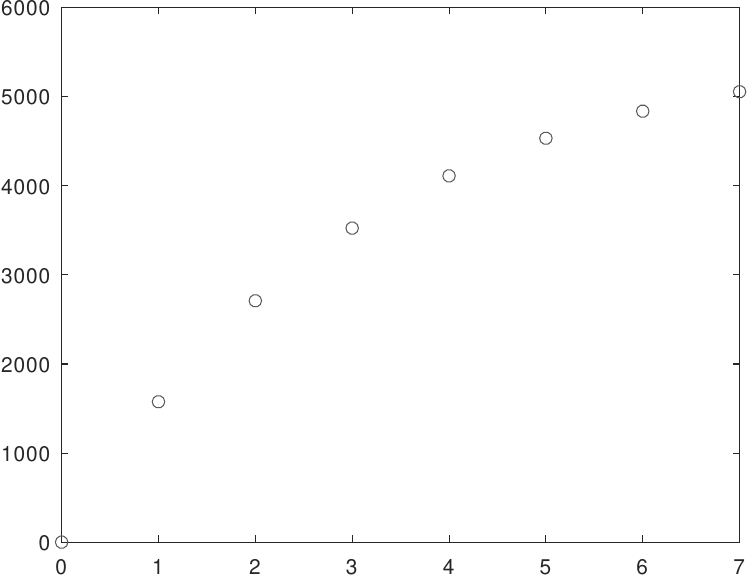}
\caption{Evolution of $x_1(t)$ (Susceptible)}
\label{fig:x1}
\end{subfigure}
\hfill
\begin{subfigure}{0.44\textwidth}
\includegraphics[width=\textwidth]{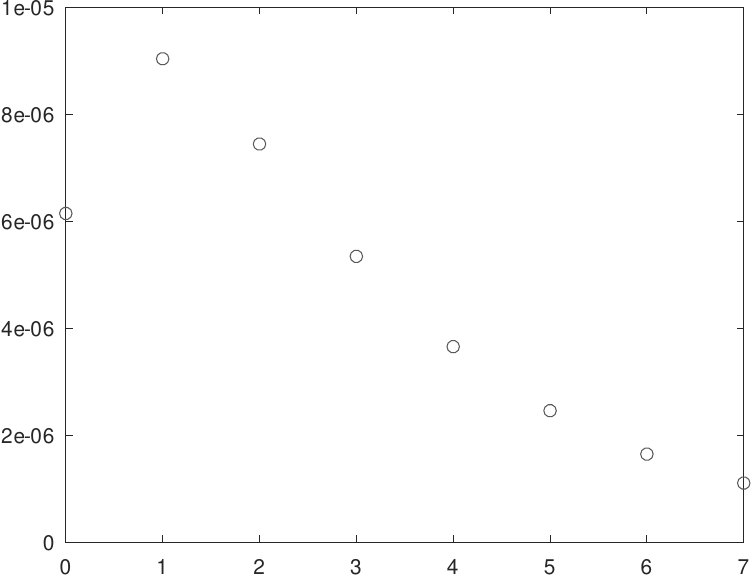}
\caption{Evolution of $x_2(t)$ (Infected)}
\label{fig:x2}
\end{subfigure}
\hfill
\begin{subfigure}{0.44\textwidth}
\includegraphics[width=\textwidth]{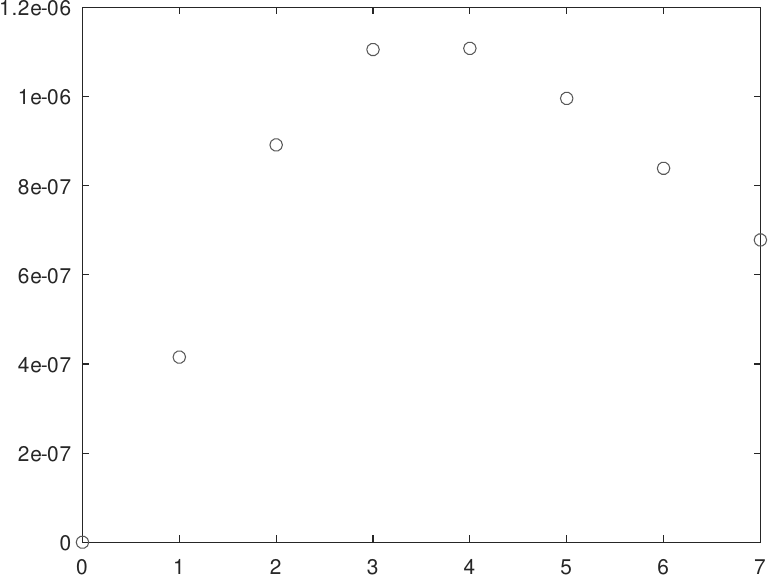}
\caption{Evolution of $x_3(t)$ (Chronic)}
\label{fig:x3}
\end{subfigure}
\hfill
\begin{subfigure}{0.44\textwidth}
\includegraphics[width=\textwidth]{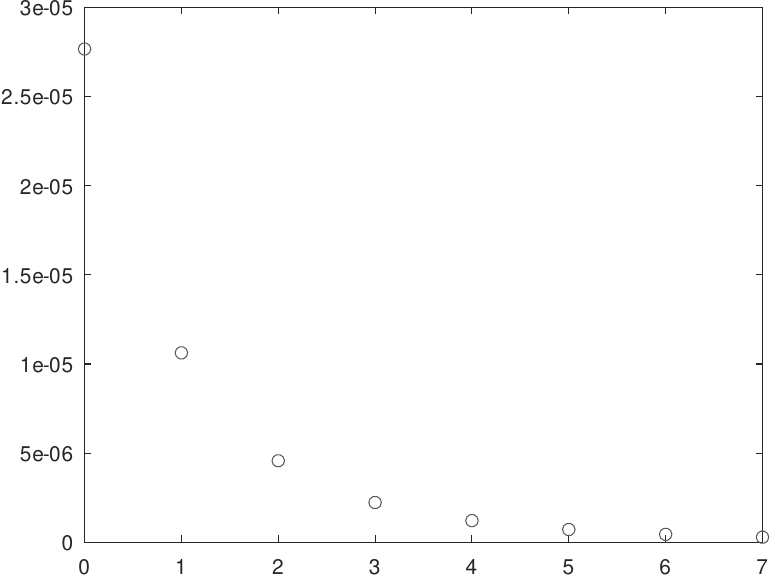}
\caption{Evolution of $x_4(t)$ (AIDS)}
\label{fig:x4}
\end{subfigure}
\caption{Example~\ref{ex01}: solution of \eqref{example} during 7 years.}
\label{fig}
\end{figure}
\end{example}


\section{Conclusion}
\label{sec5}

We have investigated the uniform stability of the singular positive
solution in an HIV/AIDS epidemic model, specifically the SICA model
on an arbitrary time scale. The purpose of incorporating time scales
is to integrate both continuous and discrete time models. We established
the permanence of each solution and, using a suitable Lyapunov function,
we derived a sufficient condition for uniform asymptotic stability of the solution.
Additionally, we presented an illustrative example to substantiate our analytical results.


\subsection*{Acknowledgments}

This research is part of first author's Ph.D. project,
which is carried out at University Mustapha Stambouli of Mascara.


\subsection*{Funding}

Torres was supported by CIDMA and is funded by the Funda\c{c}\~{a}o
para a Ci\^{e}ncia e a Tecnologia, I.P. (FCT, Funder ID = 50110000187)
under Grants UIDB/04106/2020 and UIDP/04106/2020.


\subsection*{Competing and Conflict of Interests}

The authors have no conflicts or competing of interests to declare.


\subsection*{Data Availability}

No data is associated to the manuscript.


\subsection*{Author contributions}

All authors contributed equally to this work
and reviewed and approved the final manuscript.



\bigskip


\end{document}